\documentclass[letterpaper, reqno,11pt]{article}
\usepackage[margin=1.0in]{geometry}
\usepackage{color,latexsym,amsmath,amssymb,path, amsthm}
\usepackage{graphicx}
\usepackage{amssymb}
\usepackage{enumerate}

\newtheorem{thm}{Theorem}[section]
\newtheorem{prop}[thm]{Proposition}
\newtheorem{lemma}[thm]{Lemma}
\newtheorem{cor}[thm]{Corollary}
\newtheorem{defn}[thm]{Definition}

\newtheorem{conj}[thm]{Conjecture}
\newtheorem*{thm*}{Theorem}

\usepackage{color} % this is to make the TODOs in red and the comments to Josh/Larry in green/blue
\usepackage{enumerate}

\newcommand{\CC}{\mathbb{C}}
\newcommand{\RR}{\mathbb{R}}

\newcommand{\curves}{\mathcal{L}}
\newcommand{\pts}{\mathcal{P}}
\newcommand{\eps}{\epsilon}
\newcommand{\lines}{\mathcal{L}}

\newcommand{\SSS}{\mathcal{S}} %there is a bug in latexdiff that causes problems with declaremathoperator 
\newcommand{\MMM}{\mathcal{M}}
\newcommand{\CCC}{\mathcal{C}}

\begin{document}
\title{Curves in $\RR^4$ and two-rich points}
\author{Larry Guth\thanks{Massachusetts Institute of Technology, Cambridge MA. Supported by a Simons Investigator award.} \and Joshua Zahl\thanks{Massachusetts Institute of Technology, Cambridge, MA. Supported by a NSF Postdoctoral Fellowship.}}
\date{}
\maketitle
\maketitle
\begin{abstract}
We obtain a new bound on the number of two-rich points spanned by an arrangement of low degree algebraic curves in $\RR^4$. Specifically, we show that an arrangement of  $n$ algebraic curves determines at most $C_\epsilon n^{4/3+3\epsilon}$ two-rich points, provided at most $n^{2/3+2\epsilon}$ curves lie in any low degree hypersurface and at most $n^{1/3+\epsilon}$ curves lie in any low degree surface. This result follows from a structure theorem about arrangements of curves that determine many two-rich points. 
\end{abstract}

\section{Introduction}
We will prove a new incidence bound for the number of two-rich points spanned by an arrangement of algebraic curves in $\RR^4$. This is an extension to four dimensions of a previous three-dimensional bound of the authors in \cite{GZ}. Bounds of this type were a key ingredient in the proof of the Erd\H{o}s distinct distances problem in the plane \cite{GK}, and have also been used by Ellenberg, Solymosi and the second author in \cite{SZ} to attack other planar incidence problems.

\begin{defn}[Two-rich point]
Let $\mathcal{L}$ be a set of algebraic curves in $\RR^d$. We say a point $x\in\RR^d$ is \emph{two-rich} if there are at least two curves from $\mathcal{L}$ that contain $x$. We denote the set of two-rich points by $\pts_2(\mathcal{L})$. 
\end{defn}

In \cite{GK}, Katz and the first author proved the following result about the incidence geometry of lines in $\RR^3$.

\begin{thm} \label{GKthm} If $\mathcal{L}$ is a set of $n$ lines in $\RR^3$ with at most $n^{1/2}$ lines in any plane or degree 2 surface, then $\mathcal{L}$ has at most $C n^{3/2}$ two-rich points.
\end{thm}

There are now several proofs of Theorem \ref{GKthm} and related results: \cite{K}, \cite{GZ}, and \cite{G}.  The paper \cite{K} generalizes Theorem \ref{GKthm} to any field.  The paper \cite{GZ} generalizes Theorem \ref{GKthm} further by allowing low degree algebraic curves instead of lines (also over any field).  An important open problem is to find good generalizations of Theorem \ref{GKthm} to higher dimensions.  We prove a generalization to four dimensions.  Our argument works only over $\RR$, but it applies to low degree curves and not only straight lines.  Here is our main theorem.  

\begin{thm} \label{introthm}
For every $D$ and every $\eps>0$, there is a constant $E$ so that the following holds. Let $\curves$ be an arrangement of $n$ irreducible curves of degree at most $D$ in $\RR^4$. Suppose that at most $n^{2/3+\epsilon}$ curves are contained in any three-dimensional hypersurface of degree $E$ or less, and at most $n^{1/3+2\epsilon}$ curves are contained in any two-dimensional surface of degree $100D^2$ or less. Then the number of two-rich points spanned by $\curves$ is $O_{D,\eps}(n^{4/3+3\eps})$. 
\end{thm}

\subsection{Previous work}
In \cite{SS}, Sharir and Solomon established a new incidence bound for points and lines in $\RR^4$ under the non-degeneracy condition that not too many lines lie in any plane, hyperplane, or quadric hypersurface. Sharir and Solomon establish essentially sharp bounds on the number of $k$--rich points when $k$ is large, but they do not consider the problem of bounding the number of two-rich points.  (Also, their results are only for lines and don't apply to low degree algebraic curves.)

Like Sharir and Solomon's results, Theorem \ref{twoRichPtsThm} is only proved over the reals. However, while Sharir and Solomon's result is probably extremely difficult to prove over finite fields (in particular, it would imply a sharp Szemer\'edi-Trotter bound in $\mathbb{F}_p^2$), Theorem \ref{twoRichPtsThm} is almost certainly true in finite fields, and while such a result is out of the reach of current methods, we suspect that it is much less difficult than proving a sharp analogue of the Szemer\'edi-Trotter theorem over finite fields. 

\subsection{Outline of the proof}

Our approach to Theorem \ref{introthm} is based on polynomial partitioning.  The paper \cite{G} proves (a slightly weaker version of) Theorem \ref{GKthm} using polynomial partitioning.  The argument there provides the framework for our approach, but the 4-dimensional case is much subtler.  We will first describe the framework from \cite{G} and then the new ideas.  Also, there is a crucial point of the argument where we need to work over $\CC$ and apply the incidence estimates for complex algebraic curves in $\CC^3$ from \cite{GZ}.

The argument from \cite{G} is based on induction, and to make the induction close, one actually proves a slightly stronger theorem.  The theorem says that for any set of $n$ lines in $\RR^3$, if the number of two-rich points is much larger than $n^{3/2}$, then most of the two-rich points come from a small number of low degree varieties.  To state the theorem, we will use the following notation.  If $\mathcal{V}$ is a set of varieties in $\RR^d$ and $Z\subset\RR^d$ is a (higher-dimensional) variety, then we define 

$$\mathcal{V}_Z=\{V\in\mathcal{V}\colon V\subset Z\}. $$

Generalizing the argument from \cite{G} a little, we will prove the following result:

\begin{prop} \label{Ginduct} (See Proposition \ref{generalizedGuth}) For any $D$ and $\epsilon > 0$, there are constants $C$ and $C^\prime$ so that the following holds.  If $\frak L$ is a set of $n$ irreducible curves in $\RR^3$ of degree at most $D$, then there is a set $\mathcal{Z}$ of algebraic surfaces so that

\begin{enumerate}

\item Each surface $Z \in \mathcal{Z}$ is an irreducible surface of degree at most $C^\prime$.

\item Each surface $Z \in \mathcal{Z}$ contains at least $n^{1/2 + \epsilon}$ curves of $\frak L$.

\item $|\mathcal{Z}| \le 2 n^{1/2 - \epsilon}$.

\item $| P_2(\frak L) \setminus \bigcup_{Z \in \mathcal{Z}} P_{2}(\frak L_Z) | \le C L^{3/2 + \epsilon}.$

\end{enumerate}

\end{prop}

This result is proved using polynomial partitioning and induction.  Polynomial partitioning was introduced in \cite{GK}, where the following result was proved:

\begin{thm} \label{polypart1} (\cite{GK}) If $X$ is a finite set in $\RR^d$ and $E \ge 1$, then there is a non-polynomial $P$ of degree at most $E$ so that each component of $\RR^d \setminus Z(P)$ contains at most $C_d E^{-d} |X|$ points of $X$.  
\end{thm}

A little later, a more general version of polynomial partitioning was proven, which applies not just to finite sets of points but also to finite sets of lines, or more generally to finite sets of varieties.

\begin{thm}[Polynomial partitioning for varieties; see \cite{G2}, Theorem 0.3]\label{partitioningForVarietiesThm}
Let $\Gamma$ be a set of varieties in $\RR^d$, each of which has degree at most $D$ and dimension at most $e$. For each $E \geq 1$, there is a non-zero polynomial $P$ of degree at most $E$, so that each connected component of $\RR^d\backslash Z(P)$ intersects at most $C(d, e, D) E^{e-d}|\Gamma|$ varieties from $\Gamma$.
\end{thm}

Here is the rough idea of the proof of Theorem \ref{Ginduct}.  We pick a degree $E \le C^\prime$ and apply Theorem \ref{partitioningForVarietiesThm}.  This theorem tells us that there is a polynomial $P$ of degree at most $E \le C^\prime$ so that each connected component of $\RR^3 \setminus Z(P)$ intersects not too many curves of $\mathcal{L}$.  For each connected component $\Omega$ of $\RR^3 \setminus Z(P)$, we let $\frak L_\Omega$ be the set of curves of $\frak L$ that intersect $\Omega$.  By induction, we can assume that Theorem \ref{Ginduct} holds for each $\frak L_\Omega$.  For each $\Omega$, we get a set of irreducible surfaces $\mathcal{Z}_\Omega$ of degree at most $C^\prime$.  Now we define $\mathcal{Z}_1$ to be the union of $\mathcal{Z}_\Omega$ over all the components $\Omega \subset \RR^3 \setminus Z(P)$, together with all the irreducible components of $Z(P)$.  Now $\mathcal{Z}_1$ is a set of irreducible surfaces of degree at most $C^\prime$, and a simple calculation shows that it obeys the bound on two-rich points in Theorem \ref{Ginduct}.  However, $\mathcal{Z}_1$ does not close the induction, because there are too many surfaces in $\mathcal{Z}_1$, and not all the surfaces contain enough curves of $\curves$.

To find $\mathcal{Z}$, we need to process $\mathcal{Z}_1$.  The processing in \cite{G} is a simple pruning mechanism: we let

$$ \mathcal{Z} := \{ Z \in \mathcal{Z}_1 \textrm{ so that } | \mathcal{L}_Z | \ge n^{1/2 + \eps} \}. $$

\noindent It turns out that $|\mathcal{Z}|\le 2 n^{1/2 - \eps}$, and that $\mathcal{Z}$ obeys all the desired properties and closes the induction.  We review this argument in detail in Section \ref{curvesinR3}, where we prove (a slightly more general version of) Theorem \ref{Ginduct}.

The proof of Theorem \ref{introthm} has a similar framework.  We use polynomial partitioning and induction to prove a slightly stronger result.  Here is the stronger result.

\begin{thm}\label{twoRichPtsThm}
For every $D$ and every $\eps > 0$, there are constants $C$ and $C^\prime$ so that the following holds. Let $\curves$ be an arrangement of $n$ irreducible curves of degree at most $D$ in $\RR^4$. Then there are sets $\mathcal M$ and $\mathcal S$ with the following properties. $\mathcal M$ is a set of irreducible real algebraic varieties of dimension at most three, and it has cardinality at most $n^{1/3-\eps}$. Each variety $M\in\mathcal{M}$ has degree at most $C^\prime$, and $|\curves_M|\geq n^{2/3+\eps}.$ $\mathcal S$ is a set of irreducible real algebraic varieties of dimension at most two, and it has cardinality at most $n^{2/3-2\eps}$. Each variety $S\in\mathcal{S}$ has degree at most $100D^2$, and $|\curves_S|\geq n^{1/3+2\eps}.$ Finally, the number of two-rich points occurring between pairs of curves that are not contained in some surface $M\in\mathcal M$ or $S\in\mathcal S$ is small. More precisely, we have the bound
\item \begin{equation}
       \Big|\pts_2(\curves)\ \backslash\ \Big(\bigcup_{M\in\mathcal{M}}\pts_2(M(\lines)) \cup \bigcup_{S\in\mathcal{S}}\pts_2(\curves_S)\Big)\Big|\leq Cn^{4/3+3\eps}.
      \end{equation}
\end{thm}

% Should C^\prime be E?

By Theorem \ref{partitioningForVarietiesThm}, we can choose a polynomial $P$ of degree $E \le C^\prime$ so that each component of $\RR^4 \setminus Z(P)$ intersects a controlled number of curves from $\curves$.  For each component $\Omega$ of $\RR^4 \setminus Z(P)$, we let $\curves_\Omega$ be the set of curves of $\curves$ that intersect $\Omega$.  By induction, we can assume that Theorem \ref{twoRichPtsThm} holds for each $\curves_\Omega$.  The inductive hypothesis gives us a set of 3-dimensional varieties $\mathcal{M}_\Omega$ and a set of two-dimensional varieties $\mathcal{S}_\Omega$.  We now define $\MMM_1$ to be the union of the irreducible components of $Z(P)$ together with $\bigcup_\Omega \MMM_\Omega$.  Similarly, we define $\SSS_1$ to be the union $\bigcup_\Omega \SSS_\Omega$.  A simple calculation shows that $\MMM_1$ and $\SSS_1$ satisfy the bound about two-rich points at the end of Theorem \ref{twoRichPtsThm}.  However, they don't close the induction, because there are too many varieties in $\MMM_1$ and $\SSS_1$, and each variety may not contain enough curves.

To close the induction, we have to process $\MMM_1$ and $\SSS_1$.  This processing is subtler than in three dimensions, and the processing scheme that we use is the main contribution of the paper.  This processing is not a simple pruning process, like it was in the three-dimensional case.  We introduce two new processing maneuvers.  Sometimes, we remove a 3-dimensional variety $M$ from $\MMM_1$, and add a set of two dimensional varieties, $\{ S_j \}$, to $\SSS_1$, where the $S_j$ are subvarieties of $M$.   At other times, we remove a set of two-dimensional varieties, $\{ S_j \}$, from $\SSS_1$, and add to $\MMM_1$ a 3-dimensional variety $M$ containing the $S_j$.  A key insight is that getting this replacement scheme to work involves variations of the original problem.

For example, in order to carry out the second replacement maneuver from the last paragraph, we need a variation of Proposition \ref{Ginduct} where curves are replaced by two-dimensional surfaces and two-rich points are replaced by two-rich curves -- see Proposition \ref{GZForSurfaces} below.  If $\SSS$ is a set of two-dimensional surfaces, we let $\CCC_2(\SSS)$ be the set of two-rich curves of $\SSS$, i.e.~the set of irreducible algebraic curves that lie in at least two of the surfaces of $\SSS$.  Proposition \ref{Ginduct} roughly says that for $n$ surfaces in $\RR^4$, the number of 2-rich curves is at most $n^{3/2 + \eps}$, except for the contribution coming from surfaces contained in a small number of low degree 3-dimensional varieties.  

To prove such a result for two-dimensional surfaces in $\RR^4$, it looks like a reasonable idea to intersect all the objects with a generic hyperplane $H \subset \RR^4$.  For a generic $H$, each surface $S \in \SSS$ will intersect $H$ in an irreducible curve (possibly empty).  In this way, we get a set of irreducible curves $\curves_H$ in the 3-dimensional plane $H$.  We can control the two-rich points of $\curves_H$ using Proposition \ref{Ginduct}.  But this does not allow us to control the two-rich curves of $\SSS$.  The problem is that a curve $\gamma \in \CCC_2(\SSS)$ may not intersect the plane $H$.  If $\gamma$ is a small closed curve in $\RR^4$, then most hyperplanes $H$ fail to intersect $\gamma$.  If $\CCC_2(\SSS)$ consists of many small closed curves that are spread out in $\RR^4$, then every hyperplane $H$ will intersect only a small number of these curves.

The situation improves if we switch from $\RR^4$ to $\CC^4$.  An algebraic curve $\gamma$ in $\CC^4$ intersects almost every (complex) hyperplane $H$ in $\CC^4$.  By intersecting with a hyperplane, we can reduce a question about two-rich curves of surfaces in $\CC^4$ to a problem about two-rich points of curves in $\CC^3$.  We then apply the two-rich point estimate about curves in $\CC^3$ from \cite{GZ}.  In summary, we prove a key lemma about the incidences of 2-dimensional surfaces in $\RR^4$ by using the results on curves in $\CC^3$ from \cite{GZ}.  

We were hoping that we might be able to prove a result analogous to Theorem \ref{twoRichPtsThm} in all dimensions, using polynomial partitioning and induction on the dimension, but we have not been able to do so.  The problem is that we use an incidence theorem in $\CC^3$ to prove an incidence theorem in $\RR^4$.  If we had a similar incidence theorem in $\CC^4$, the tools in this paper would probably lead to an incidence theorem in $\RR^5$.  However, we don't know how to prove such an incidence theorem in $\CC^4$.  

In a broader sense, the problem is that different tools work well in different fields.  Polynomial partitioning works over $\RR$.  But many tools in algebraic geometry work better over $\CC$ because $\CC$ is algebraically closed.  In particular, intersecting a variety with a hyperplane works better over $\CC$.  (A similar tension appears in \cite{SS}, where polynomial partitioning plays a crucial role, but some parts of the argument are carried out over $\CC$.)  The second author has adapted polynomial partitioning arguments to the complex setting in certain situations in \cite{ShZ} (joint with Sheffer and Szab\'o) and in \cite{Z2}, but these ideas are not yet enough to adapt the main argument in this paper to $\CC$.  In the last section of the paper, we share some speculations and failed attempts to get polynomial partitioning to work over $\CC$.

\subsection{Thanks}
The authors would like to thank Misha Rudnev for pointing out an error in an earlier version of this manuscript.

\section{Notation and background}
\subsection{Notation}
We write $A=O(B)$ or $A\lesssim B$ to mean $A\leq C B$ for some absolute constant $C$. If the constant is allowed to depend on a set of parameters $t_1,\ldots,t_\ell$, then we will write $A=O_{t_1,\ldots,t_\ell}(B)$ or $A\lesssim_{t_1,\ldots,t_\ell}B$. %If we wish to refer to this implicit constant later, we will write $A\leq C_{t_1,t_2}B$. This means that the constant $C_{t_1,t_2}$ depends only on the parameters $t_1$ and $t_2$. 
\subsection{Real algebraic geometry}
\begin{defn}[Degree and dimension of a real variety]\label{defnOfRealDegree}
Let $V\subset\RR^d$ be a real algebraic variety, and let $V^*$ be the smallest complex variety in $\CC^d$ that contains $V$. We define the degree of $V$ to be the degree of $V^*$; the latter is the sum of the degrees of the irreducible components of $V^*$. For the dimension of a real variety we refer the reader to \cite{BCR}. Informally, however, the dimension of a real algebraic variety is the largest integer $d^\prime$ so that the variety contains subset that is homeomorphic to the open unit cube $(0,1)^{d^\prime}$.
\end{defn}

Some of the results about real varieties that we will use do not refer to the degree of a variety. Instead, they refer to the number and degree of the polynomials needed to define the variety. Results of this type begin with hypotheses such as ``let $M$ be a real variety that can be defined by $a_1$ polynomials of degree at most $a_2$.'' If $M\subset\RR^d$ has degree $D$, then $a_1,a_2 = O_{D,d}(1)$. Similarly, if $M$ can be defined by $a_1$ polynomials of degree at most $a_2$, then the degree of $M$ is $O_{a_1,a_2,d}(1)$. To keep our notation consistent, we will quote these results by specifying the degree of the varieties involved. For our purposes, this will be an equivalent formulation. 

The following theorem describes the number of (Euclidean) connected components that can be ``cut out'' by a real polynomial.
\begin{thm}[Barone and Basu, \cite{BB}]\label{BBThm}
Let $M\subset\RR^d$ be a variety of degree $D$ and dimension at most $d^\prime$. Let $P\in\RR[x_1,\ldots,x_d]$ be a polynomial. Then both $M\cap Z(P)$ and $M\backslash Z(P)$ contain $O_{d,d^\prime,D}(d^\prime)^{\deg(P)}$ connected components.
\end{thm}
Observe that in the special case $M=d$, Theorem \ref{BBThm} states that the number of connected components of $\RR^d\backslash Z(P)$ is $O_{d}(d^{\deg P})$. This is known as the Milnor-Thom theorem.

\begin{prop}\label{complexityOfProjections}
Let $\gamma\subset\RR^d$ be a real algebraic variety of degree $D$. Let $\pi\colon\RR^d\to\RR^{d^\prime}$ be the projection to the first $d^\prime$ coordinates. Then the Zariski closure of $\pi(\gamma)$ is an algebraic variety of degree $O_{d,D}(1)$ and dimension at most $\dim(\gamma)$.
\end{prop}
\subsection{Polynomial partitioning}

Theorem \ref{partitioningForVarietiesThm} will play an important role in the paper.  We recall the statement here.

\begin{thm*}[Polynomial partitioning for varieties; see \cite{G2}, Theorem 0.3]
Let $\Gamma$ be a set of varieties in $\RR^d$, each of which has degree at most $D$ and dimension at most $e$. For each $E\geq 1$, there is a non-zero polynomial $P$ of degree at most $E$, so that each connected component of $\RR^d\backslash Z(P)$ intersects $O_{d, D}(E^{e-d}|\Gamma|)$ varieties from $\Gamma$.
\end{thm*}

Combining Theorem \ref{partitioningForVarietiesThm} and Proposition \ref{complexityOfProjections}, we obtain the following corollary.
\begin{cor}\label{corOfVarPartThm}
Let $\Gamma$ be a set of varieties in $\RR^d$, each of which has degree at most $D$ and dimension at most $e$. For each $E\geq 1$ and $e\leq d^\prime\leq d$, there is a non-zero polynomial $P\in\RR[x_1,\ldots,x_d]$ of the form $P(x_1,\ldots,x_d) = Q(x_1,\ldots, x_{d^\prime})$ of degree at most $E$, so that each connected component of $\RR^d\backslash Z(P)$ intersects $O_{d,D}(E^{e-d^\prime}|\Gamma|)$ varieties from $\Gamma$.
\end{cor}

\section{Warm-up: two-rich points in three dimensions} \label{curvesinR3}
As a warm-up, we will first prove a bound on the number of two-rich points spanned by an arrangement of curves that are contained in a low degree three dimensional real variety. We will closely follow the proof from \cite{G}.
%\begin{prop}
 \begin{prop}\label{generalizedGuth}
For each $d\geq 3$, $D,E\geq 1$ and $\epsilon>0$, there are constants $C$ and $C^\prime$ so that the following holds. Let $M\subset\RR^d$ be an irreducible three-dimensional variety of degree at most $E$. Let $\mathcal{L}$ be a set of irreducible curves in $M$, each of which has degree at most $D$. Then for each $\epsilon>0$ there is a set $\mathcal{S}$ of at most $n^{1/2-\epsilon}$ irreducible varieties, each of dimension at most two, so that each $S\in\mathcal{S}$ has degree at most $C^\prime$ and satisfies $|\lines_S| \geq n^{1/2+\epsilon}$. Furthermore, there are few two-rich points not covered by the surfaces in $\mathcal{S}$. More precisely, we have the estimate
\begin{equation}\label{threeDimTwoRichPtsBd}
\Big|\pts_2( \lines )\ \backslash\ \bigcup_{S\in\mathcal{S}}\pts_2(\lines_S)\Big| \leq Cn^{3/2+\epsilon}.
\end{equation}
\end{prop}
\begin{proof}

We will prove the result by induction on $n$; the case when $n$ is small is trivial, provided we select $C$ sufficiently large.  After a rotation, we can assume that $M$ intersects every affine hypersurface of the form $P(x_1,x_2,x_3)=0$ properly (e.g. the intersection has dimension strictly smaller than three). Use Corollary \ref{corOfVarPartThm}  to select a polynomial $P\in\RR[x_1,\ldots,x_d]$ of the form $P(x_1,\ldots,x_d)=P(x_1,x_2,x_3)$ of degree $E_1$ so that each connected component of $\RR^d\backslash Z(P)$ intersects $O_{D,d}(nE_1^{-2})$ curves from $\mathcal{L}$. We will select the parameter $E_1$ later. 

For each cell $\Omega$ of this partition, let $\mathcal{L}_{\Omega}$ be the set of curves from $\mathcal{L}$ that meet $\Omega$. Apply the induction hypothesis to each set $\mathcal{L}_{\Omega}$; let $\mathcal{S}_\Omega$ be the resulting set of irreducible surfaces. Let $\mathcal{S}_1= \bigcup_{\Omega}\mathcal{S}_\Omega$. Recall that $M\cap Z(P)$ is a proper intersection, and thus $M\cap Z(P)$ is an algebraic variety of dimension at most two and degree $O_{d,E_1,\epsilon}(1)$. Note that $|\mathcal{S}_{\Omega}|\lesssim_{d,D}(nE_1^{-2})^{1/2-\epsilon}$ for each index $\Omega$, and thus 
$$
|\mathcal{S}_1|\lesssim_{d,D} E_1^{2}n^{1/2-\epsilon}.
$$
Observe that for each cell $\Omega$, we have
\begin{equation*}
\begin{split}
\Big| \pts_2(\mathcal{L}_\Omega)\ \backslash \bigcup_{S\in\mathcal{S}_\Omega}\pts_2(\mathcal{L}_S)\Big|&\leq C|\mathcal{L}_\Omega|^{3/2+\epsilon}\\
&\lesssim_{d,D} C(nE_1^{-2})^{3/2+\epsilon},
\end{split}
\end{equation*}
and thus if we define $S_0=Z(P)\cap M$, then
\begin{equation}\label{boundOnPts2ThreeDim}
 \begin{split}
  \Big|\pts_2(\mathcal{L})\ \big\backslash\ \Big(\pts_2(\mathcal{L}_{S_0})\cup\bigcup_{S\in\mathcal{S}_1}\pts_2(\mathcal{L}_S)\Big) \Big|&\leq \sum_\Omega \Big| \pts_2(\mathcal{L}_\Omega)\ \backslash \bigcup_{S\in\mathcal{S}_\Omega}\pts_2(\mathcal{L}_S)\Big|\\
  &\lesssim_{d,D} C E_1^{-2\epsilon}n^{3/2+\epsilon}.
 \end{split}
\end{equation}
If we choose $E_1=O_{d,D,\epsilon}(1)$ sufficiently large, then $\eqref{boundOnPts2ThreeDim}\leq \frac{C}{3}n^{3/2+\epsilon}$. Note that with such a choice of $E_1$, $Z(P)\cap M$ has degree $O_{d,D,E,\epsilon}(1)$. Thus if we choose $C^\prime$ sufficiently large (depending only on $d,D,E$ and $\epsilon$), then $S_0$ is a union of $O_{d,D,E,\epsilon}(1)$ irreducible varieties, each of dimension at most two and degree at most $C^\prime$. Let $\mathcal{S}_2$ be the union of $\mathcal{S}_1$ and the irreducible components of $S_0$. Then $|\mathcal{S}_2|\lesssim_{d,D,E,\epsilon}n^{1/2-\epsilon}$.

Define 
$$
\mathcal{S}=\{S\in\mathcal{S}_2\colon |\mathcal{L}_S|\geq 2 n^{1/2+\epsilon}\}.
$$
Since each pair of curves can intersect $O_{d,D}(1)$ times, we conclude that 
\begin{equation}\label{boundOnTwoRichDeletedSurfaces}
 \begin{split}
  \sum_{S\in\mathcal{S}_1\backslash\mathcal{S}_2}|\pts_2(\mathcal{L}_S) |&\lesssim_{d,D} |\mathcal{S}_1|( n^{1/2+\epsilon})^2\\
  &\lesssim_{d,D} E_1^2 n^{1/2-\epsilon}( n^{1/2+\epsilon})^2\\
  &\lesssim_{d,D,\epsilon} n^{3/2+\epsilon}.
 \end{split}
\end{equation}
If the constant $C=O_{d,D,E,\epsilon}(1)$ from the statement of Proposition \ref{generalizedGuth} is selected sufficiently large, then $\eqref{boundOnTwoRichDeletedSurfaces}\leq \frac{C}{3}n^{3/2+\epsilon}$. Thus \eqref{threeDimTwoRichPtsBd} holds for this choice of $\mathcal{S}$. Next, we will verify that that $|\mathcal{S}|\leq n^{1/2-\epsilon}$. To do this, we will first need the following lemma. 
\begin{lemma}[Unions of surfaces contain many lines]\label{unionsOfSurfacesLem}
For each $d,D\geq 1$, there is a constant $C_1$ so that the following holds. Let $\mathcal{L}$ be a set of curves in $\RR^d$, each of degree at most $D$. Let $\mathcal{S}$ be a set of two-dimensional surfaces in $\RR^d$, each of degree at most $D$. If each surface from $\mathcal{S}$ contains at least $C_1|\mathcal{S}|$ curves from $\mathcal{L}$, then 
\begin{equation}
\sum_{S\in\mathcal{S}} |\mathcal{L}_S|\leq 2\big|\bigcup_{S\in\mathcal{S}}\mathcal{L}_S\big|.
\end{equation}
\end{lemma}
\begin{proof}
First, note that any two surfaces from $\mathcal{S}$ can intersect in at most $O_{d,D}(1)$ irreducible curves, and thus the intersection can contain at most $O_{d,D}(1)$ curves from $\mathcal{L}$. Let $S_1,\ldots,S_t,\ t=|\mathcal{S}|$ be an enumeration of the curves from $\mathcal{S}$. By inclusion-exclusion, we have 
\begin{equation*}
 \begin{split}
\Big|\bigcup_{i=1}^{t}\mathcal{L}_{S_i}\Big| &\geq \sum_{i=1}^t \Big(|\mathcal{L}_{S_i}|-\sum_{j=1}^{i} |\mathcal{L}_{S_i} \cap \mathcal{L}_{S_j}|\Big)\\
&\geq \sum_{i=1}^t \big(|\mathcal{L}_{S_i}|-O_{d,D}(1)|\mathcal{S}|\big)\\
&\geq\frac{1}{2}\sum_{i=1}^t |\mathcal{L}_{S_i}|,
\end{split}
\end{equation*}
provided $C_1=O_{d,D}(1)$ is chosen sufficiently large.
\end{proof}

Apply Lemma \ref{unionsOfSurfacesLem} to $\mathcal{S}$; each surface $S\in\mathcal{S}_2$ contains at least $2n^{1/2+\epsilon}$ curves from $\mathcal{L}$, and this is larger than $C_1|\mathcal{S}|$ provided $n$ is sufficiently large compared to $d,D,$ and $\epsilon$. We conclude that $|\mathcal{S}|\leq n^{1/2-\epsilon}.$ This completes the proof of Proposition \ref{generalizedGuth}.
\end{proof}

We note the following corollary of Lemma \ref{unionsOfSurfacesLem}.
\begin{cor}\label{corUnionsOfSurfacesLem}
 Let $\mathcal{L}$ be a set of curves in $\RR^d$, each of degree at most $D$. Let $\mathcal{S}$ be a set of two-dimensional surfaces in $\RR^d$, each of degree at most $D$. Suppose that each surface from $\mathcal{S}$ contains at least $A$ curves from $\mathcal{L}$. Then
\begin{equation}
 \big|\bigcup_{S\in\mathcal{S}}\mathcal{L}_S\big|\gtrsim_{d,D}\min\big(A^2,\ A|\mathcal{S}|\big).
\end{equation}

\end{cor}

\section{Incidence bounds coming from $\CC^3$}\label{bipartiteGZ}

In this section, we prove incidence bounds on two-dimensional surfaces in $\RR^4$ that are based on incidence bounds for curves in $\CC^3$.  The key input is the following bound on the number of two-rich points spanned by a collection of curves in $\CC^3$. This is Lemma 12.2 from \cite{GZ}.  (It is a small variation of the main theorem of \cite{GZ}, Theorem 1.2.)
%[todo: put everything in a bdd degree three-dimensional variety,not in $K^3$].
\begin{prop}\label{GZProp}
Let $D>0$. Then there are constants $C_1',C_2'$ so that the following holds. Let $\mathcal{L}$ be an arrangement of irreducible curves in $\CC^3$, each of degree at most $D$. Then for each number $A>C_1' n^{1/2}$, at least one of the following two things must occur
\begin{itemize}
 \item There is a curve $\gamma \in \mathcal{L}$ that contains at most $C_2' A$ two-rich points of $\mathcal{L}$.
 \item There is an irreducible surface $Z\subset \CC^3$ of degree at most $100 D^2$ that contains at least $A$ curves from ${\mathcal{L}}$. %Furthermore, $\hat Z$ is doubly ruled by curves of degree at most $D$.
\end{itemize}
\end{prop}
%\TODO{we'll need to modify \cite{GZ} to include this statement.}

Here is our main result on surfaces in $\RR^4$.  

\begin{prop}[Two-rich curves for surfaces in $\RR^4$]\label{GZForSurfaces}
For each $D\geq 1$, there are constants $C_1,C_2$ so that the following holds. Let $\mathcal{S}$ be a collection of $n$ irreducible two-dimensional surfaces in $\RR^4$, each of degree at most $D$, and let $A\geq C_1 n^{1/2}$. Then there exists a set $\mathcal{M}$ of at most $n/A$ three-dimensional varieties, each of degree at most $100D^2$, such that each variety contains $\geq A$ surfaces from $\mathcal{S}$. Furthermore,
\begin{equation*}
\sum_{\gamma\in\mathcal{C}_2(\mathcal{S})}|\{S\in \mathcal{S}^\prime\colon \gamma\subset S \}|\leq C_2An,
\end{equation*}
where $\mathcal{C}_2(\mathcal{S})$ is the set of irreducible one-dimensional curves contained in two or more surfaces from $\mathcal{S}$, and 
\begin{equation*}
\mathcal{S}^\prime=\mathcal{S}\ \backslash\bigcup_{M\in \mathcal{M}}\{S\in\mathcal{S}\colon S\subset M\}.
\end{equation*}
\end{prop}

Proposition \ref{GZForSurfaces} is a corollary of the following lemma.

\begin{lemma}\label{manyIncidencesOneSurface}
For each $D\geq 1$, there are constants $C_1,C_2$ so that the following holds. Let $\mathcal{S}$ be a collection of $n$ irreducible two-dimensional surfaces in $\RR^4$, each of degree at most $D$, and let $A\geq C_1 n^{1/2}$. Then at least one of the following must hold
\begin{itemize}
 \item[(A)] There is a three-dimensional variety $M\subset\RR^4$ of degree at most $100D^2$ which contains at least $A$ surfaces from $\mathcal{S}$.
 \item[(B)] The surfaces in $\mathcal{S}$ determine few rich curves. More precisely,
 \begin{equation*}
\sum_{\gamma\in\mathcal{C}_2(\mathcal{S})}|\{S\in \mathcal{S}\colon \gamma\subset S \}|\leq C_2An.
\end{equation*}
\end{itemize}
\end{lemma}
\begin{proof}[Proof of Proposition \ref{GZForSurfaces} using Lemma \ref{manyIncidencesOneSurface}]
Let $\mathcal{M}_0=\emptyset$ and let $\mathcal{S}_0=\mathcal{S}$. For each $j=0,1,\ldots,$ apply Lemma \ref{manyIncidencesOneSurface} to $\mathcal{S}_j$ with the value of $A$ from the statement of Proposition \ref{GZForSurfaces}. If (A) holds, let $\mathcal{M}_{j+1} = \mathcal{M}_j\cup\{M\}$, where $M$ is the three dimensional variety given by Lemma \ref{manyIncidencesOneSurface}, and let $\mathcal{S}_{j+1}=\{S\in \mathcal{S}_j\colon S\not\subset M\}.$ Note that each $M\in\mathcal{M}_{j+1}$ contains at least $A$ surfaces from $\SSS$.  Also note that $|\SSS_{j+1}| \le | \SSS_j| - A$, and so this process can continue at most $n/A$ times, until (A) must fail. At this point (B) holds, and we are done. 
\end{proof}
It remains to prove Lemma \ref{manyIncidencesOneSurface}. Lemma \ref{manyIncidencesOneSurface} is a consequence of the following lemma.
\begin{lemma}\label{allSurfacesSameNumberTwoRichPts}
For each $D\geq 1$, there are constants $C_1,C_2$ so that the following holds. Let $\mathcal{S}$ be a collection of $n$ irreducible two-dimensional surfaces in $\RR^4$, each of degree at most $D$, and let $A\geq C_1 n^{1/2}$. Suppose that for each $S\in\mathcal{S}$, there are at least $C_2A$ distinct irreducible curves $\gamma\subset S$ that are incident to at least one other surface from $\mathcal{S}$. Then there is an irreducible three dimensional variety $M$ of degree at most $100D^2$ that contains $\geq A$ surfaces from $\mathcal{S}$. 
\end{lemma}
\begin{proof}[Proof of Lemma \ref{manyIncidencesOneSurface} using Lemma \ref{allSurfacesSameNumberTwoRichPts}]
We will prove Lemma \ref{manyIncidencesOneSurface} by induction on $n$, for all $n\leq C_1^{-2}A^2$. If $n=1$ then $\mathcal{C}_2(\mathcal{S})$ is empty, so conclusion (B) of Lemma \ref{manyIncidencesOneSurface}  holds automatically and we are done. Now suppose Lemma \ref{manyIncidencesOneSurface} has been proved for all sets of surfaces of size at most $n-1$. Applying Lemma \ref{allSurfacesSameNumberTwoRichPts}, we conclude that either there is an irreducible three dimensional variety $M$ of degree at most $100D^2$ that contains at least $A$ surfaces from $\mathcal{S}$, or there exists a surface $S_0\in\mathcal{S}$ for which there exists fewer than $C_2A$ distinct irreducible curves $\gamma\subset S$ that are incident to at least one other surface from $\mathcal{S}$. If the former happens then conclusion (A) holds and we are done. If the latter happens, let $\mathcal{S}^\prime=\mathcal{S}\backslash \{S_0\}$. Then $|\mathcal{S}^\prime|=n-1$, so we can apply the induction hypothesis. Either conclusion (A) holds, or there does not exist a three dimensional variety of degree at most $100D^2$ that contains at least $A$ surfaces from $\mathcal{S}^\prime$. This implies that
\begin{equation*}
 \sum_{\gamma\in\mathcal{C}_2(\mathcal{S}^\prime)}|\{S\in \mathcal{S}^\prime\colon \gamma\subset S \}|\leq C_2A(n-1).
\end{equation*}
Therefore
\begin{equation}
 \begin{split}
  \sum_{\gamma\in\mathcal{C}_2(\mathcal{S})}|\{S\in \mathcal{S}\colon \gamma\subset S \}|&< C_2A + \sum_{\gamma\in\mathcal{C}_2(\mathcal{S}^\prime)}|\{S\in \mathcal{S}^\prime\colon \gamma\subset S \}|\\
  &<C_2A + C_2A(n-1)\\
  &= C_2An.
 \end{split}
\end{equation}
Thus conclusion (B) holds.

\end{proof}

The proof of Lemma \ref{allSurfacesSameNumberTwoRichPts} uses a principle known as degree reduction. This is the phenomenon that if a set of varieties intersects much more frequently than one would expect, then  this set of varieties has algebraic structure. More precisely, the set of varieties can be contained in the zero-set of a polynomial of lower degree than one might expect using dimension counting arguments. We will use the following degree reduction type result.
\begin{lemma}[Degree reduction]\label{degreeReductionLem}
For each $D\geq 1$, there are constants $C_3,C_4$ so that the following holds. Let $\mathcal{S}$ be a set of $n$ irreducible two dimensional surfaces, each of which has degree at most $D$. Suppose that $A\geq C_3 n^{1/2}$ and that for each surface $S\in\mathcal{S}$, there are at least $C_4A$ irreducible curves $\gamma\subset S$ that are contained in some other surface $S^\prime\subset S$. Then there exists a polynomial $P\in\RR[x_1,\ldots,x_4]$ of degree at most $n/(2A)$ whose zero-set contains every surface in $\mathcal{S}$. 
\end{lemma}
See Proposition 12.4 from \cite{GZ} for details. Proposition 12.4 deals with irreducible curves and points (rather than surfaces and curves), but the argument is identical.

\begin{proof}[Proof of Lemma \ref{allSurfacesSameNumberTwoRichPts}]
Use Lemma \ref{degreeReductionLem} to find a polynomial $P$ of degree at most $n/(2A)$ whose zero-set contains every surface $S\in\mathcal{S}$. Write $P$ as a product of irreducible factors $P=P_1\ldots P_\ell$, and for each index $j$, define 
\begin{equation*}
 \mathcal{S}_j=\{S\in\mathcal{S}\colon S\subset Z(P_j)\}.
\end{equation*}
Note that for each index $j$ and each $S\in\mathcal{S}_j$, we have
\begin{equation*}
|\{\gamma\in\mathcal{C}_2(\mathcal{S})\colon \gamma\subset S,\ \gamma\subset S^\prime\ \textrm{for some}\ S^\prime\in \mathcal{S}_{j^\prime},\ j^\prime\neq j\}|\lesssim_D \deg P,
\end{equation*}
since each curve in the above set is an irreducible component of $S\cap Z(P/P_j)$. By Theorem \ref{BBThm}, the number of irreducible components is $\lesssim_D \deg P$. Thus if we select $C_1$ (from the statement of Lemma \ref{allSurfacesSameNumberTwoRichPts}) sufficiently large (depending only on $D$), then 
\begin{equation*}
 |\{\gamma\in\mathcal{C}_2(\mathcal{S})\colon \gamma\subset S,\ \gamma\subset S^\prime\ \textrm{for some}\ S^\prime\in \mathcal{S}_{j^\prime},\ j^\prime\neq j\}|\leq C_2A/2,
\end{equation*}
and thus for each index $j$ and each $S\in\mathcal{S}_j$, there are at least $C_2A/2$ irreducible curves $\gamma\subset S$ that are contained in at least one other surface $S^\prime\in\mathcal{S}_j$.

Recall that $\sum_j |\mathcal{S}_j|\geq n$ and $\sum_j \deg P_j \leq n/(2A)$. Let $\mathcal{J}=\{j=1,\ldots,\ell\colon |\mathcal{S}_j|\geq A\}$.  To finish the proof, we need to find some $j_0 \in \mathcal{J}$ so that the degree of $P_{j_0}$ is at most $100 D^2$.  

Since $\ell\leq\deg P = n/(2A)$, we have $\sum_{j\in\mathcal{J}}|\SSS_j| \geq n-A(n/(2A))\geq n/2$. By pigeonholing, there exists an index $j_0 \in \mathcal{J}$ so that
\begin{equation}
\begin{split}
|\mathcal{S}_{j_0}|&\geq \frac{1}{2}\frac{n}{\deg P}{\deg P_{j_0}}\\
&\geq \frac{1}{2}\frac{n}{(\deg P)^2}(\deg P_{j_0})^2\\
&\geq \frac{1}{2}\frac{n}{(n/(2A))^2}(\deg P_{j_0})^2\\
&\geq\frac{A^2}{n}(\deg P_{j_0})^2.
\end{split}
\end{equation}
%See the proof of Proposition 12.4 from \cite{GZ} for details.

Consider the complex variety $Z_{\CC}(P_{j_0})\subset\CC^4$; this variety contains the complexification of each surface $S\in\mathcal{S}_{j_0}$. Let $H\subset\CC^4$ be a hyperplane that is generic with respect to $P_{j_0}$, and $\mathcal{S}_{j_0}$. Then $H\cap Z_{\CC}(P_{j_0})$ is an irreducible two dimensional variety in $H$.  The complexification $S^*$ of each surface $S\in\mathcal{S}_{j_0}$ meets $H$ in an irreducible curve.  Let $\mathcal{L}_{j_0}$ be the set of all these irreducible curves $S^* \cap H$, where $S \in \mathcal{S}_{j_0}$.  For each curve $S^* \cap H \in \mathcal{L}$, there are at least $C_2 A/2$ points in $S^*\cap H$ that are contained in $(S')^* \cap H$ for at least one other surface $S^\prime\in\mathcal{S}_{j_0}$.  (In this step, we used crucially that we are working over $\CC$!  If $S^*$ and $(S')^*$ intersect in a curve $\gamma \subset\CC^4$, then for a generic $H$, $\gamma \cap H$ will be non-empty.)  

This is the setup to apply Proposition \ref{GZProp}.  If $C_1', C_2'$ are the constants in Proposition \ref{GZProp}, then we choose $C_1 = C_1'$ and $C_2 = 2 C_2'$.  Each curve of $\mathcal{L}$ contains at least $C_2' A$ two-rich points of $\mathcal{L}$.  Also, the number of curves in $\mathcal{L}$ is $| \SSS_{j_0} | \le n$, and so $A > C_1' |\mathcal{L}|^{1/2}$.  
By Proposition \ref{GZProp}, there is an irreducible two-dimensional surface $Z \subset H$ of degree at most $100 D^2$ that contains $S^* \cap H$ for at least $A$ different surfaces $S \in \SSS_{j_0}$.  

We claim that $Z$ is actually equal to $Z_{\CC}(P_{j_0}) \cap H$.  We know that $Z \cap Z_{\CC}(P_{j_0}) \cap H$ contains at least $A$ different irreducible curves.  But $\deg Z \le 100 D^2$ and $\deg P_{j_0} \le n/A \le C_1^{-1} n^{1/2}$, and so $(\deg Z) (\deg P_{j_0}) < A$.  Since $Z$ is irreducible, B\'ezout's theorem (cf. Theorem 5.7 in \cite{GZ}) implies that $Z = Z_{\CC}(P_{j_0}) \cap H$.  Therefore $\deg (Z_{\CC}(P_{j_0}) \cap H) \le 100 D^2$.  

Since $\deg (Z_{\CC}(P_{j_0}) \cap H) \le 100 D^2$ for a generic hyperplane $H$, $\deg P_{j_0} \le 100 D^2$.  (Recall that the degree of a $k$-dimensional variety in $\CC^n$ is the number of intersection points with a generic $(n-k)$-plane cf. Definition 18.1 in \cite{Harris}.  From this it follows that if $\deg Z(P_{j_0}) \cap H \le 100 D^2$ for a generic hyperplane $H$, then $\deg Z(P_{j_0}) \le 100 D^2$ too.  Finally since $P_{j_0}$ is irreducible (and hence square-free), $\deg P_{j_0} = \deg Z(P_{j_0})$.)
 
\end{proof}

\subsection{Degree bounds for surfaces}

The results from \cite{GZ} also lead to degree bounds for two-dimensional surfaces that contain a set of low-degree curves with many two-rich points.  In particular, we will use the following result.  

\begin{lemma}\label{manyTwoRichImpliesBddDegree}
For each $D,E\geq 1$, there is a constant $C_1$ so that the following holds. Let $S\subset \CC^3$ be an irreducible surface of degree at most $E$. Let $\mathcal{L}$ be a set of of $n$ irreducible curves of degree at most $D$ that are contained in $S$. Suppose $\pts_2(\mathcal{L})\geq C_1n$. Then $S$ has degree at most $100D^2$.
\end{lemma}
\begin{proof}
Let $\mathcal{L}^\prime\subset\mathcal{L}$ be the set of curves that intersect $\pts_2(\mathcal{L})$ in at least $C_1/2$ points. By B\'ezout's theorem, any two curves from $\mathcal{L}$ can intersect in at most $D^2$ points, and thus each curve from $\mathcal{L}$ can intersect $\pts_2(\mathcal{L})$ in at most $D^2n$ places. We conclude that $|\mathcal{L}^\prime|\geq C_1/D^2$. 

We now apply Proposition 10.2 from \cite{GZ}. This proposition says that there is a constant $C$ (depending only on $D$) so that if $S$ contains $\geq CE^2$ curves $\gamma$, and on each such curve there are $\geq CE$ points of intersection with a curve of degree $\leq D$ that is contained in $S$ (the curve must be distinct from $\gamma$), then there is a Zariski-open subset $O\subset S$ so that for each $z\in O$, there exist at least two distinct curves of degree $\leq D$ that pass through $z$ and are contained in $S$. If we select $C_1>D^2CE^2$, then the conclusion of Proposition 10.2 applies. By Proposition 3.4 from \cite{GZ}, this implies that the degree of $S$ is at most $100D^2$. 
\end{proof}

\begin{cor}
For each $D,E\geq 1$, there is a constant $C_1$ so that the following holds. Let $S\subset \CC^4$ be an irreducible (two-dimensional) surface of degree at most $E$. Let $\mathcal{L}$ be a set of of $n$ irreducible curves of degree at most $D$ that are contained in $S$. Suppose $\pts_2(\mathcal{L})\geq C_1n$. Then $S$ has degree at most $100D^2$.
\end{cor}
\begin{proof}
The corollary follows by applying a generic (with respect to $S$ and $\mathcal{L})$ linear transformation $\CC^4\to\CC^3$. This transformation preserves the degree of $S$ and all of the curves in $\mathcal{L}$. We apply Lemma \ref{manyTwoRichImpliesBddDegree} to the image of $S$, and conclude that the degree of $S$ is at most $100D^2$.  
\end{proof}

\begin{cor}\label{manyTwoRichDecomposition}
For each $D,E\geq 1$, there is a constant $C_1$ so that the following holds. Let $S\subset \RR^4$ be an irreducible (two-dimensional) surface of degree at most $E$. Let $\mathcal{L}$ be a set of of $n$ irreducible curves of degree at most $D$ that are contained in $S$. Suppose $\pts_2(\mathcal{L})\geq C_1n$. Then $S$ has degree at most $100D^2$.
\end{cor}

\section{Proof of Theorem \ref{twoRichPtsThm} }\label{mainProofSection}
\subsection{Polynomial partitioning and the induction hypothesis}
We will prove Theorem \ref{twoRichPtsThm} by induction on $n$. The case where $n$ is small (compared to $D$ and $\epsilon$) is trivial, provided we choose the constant $C=O_{D,\epsilon}(1)$ from the statement of Theorem \ref{twoRichPtsThm} to be sufficiently large. Fix a number $E=O_{D,\eps}(1)$. Use Theorem \ref{partitioningForVarietiesThm} to find a polynomial $P$ of degree at most $E$ so that $O_{D}(n E^{-3})$ curves from $\lines$ intersect each cell $\Omega$.  Also recall that by Theorem \ref{BBThm}, there are $O(E^4)$ cells.  Let $\curves_\Omega \subset \curves$ be the set of curves that intersect the cell $\Omega$.  Apply the induction hypothesis to $\curves_\Omega$ for each cell $\Omega$. We obtain sets $\mathcal{M}_{\Omega}$ and $\mathcal{S}_\Omega$ with the following properties.  First, the number of three-dimensional varieties obeys the bound
$$
|\mathcal{M}_{\Omega}|\lesssim_D (nE^{-3})^{1/3-\eps}.
$$ 
For each $M\in \mathcal{M}_{\Omega},$ 
$$|\curves_{M,\Omega}|\gtrsim_D (nE^{-3})^{2/3+\eps}.$$
 
 The number of two-dimensional varieties obeys the bound
$$
|\mathcal{S}_{\Omega}|\lesssim_D (nE^{-3})^{2/3-2\eps}.
$$
For each $S\in \mathcal{S}_{\Omega},$ 
$$
|\curves_{S,\Omega}|\gtrsim_D (nE^{-3})^{1/3+2\eps}.
$$
Finally, we have
\begin{equation}\label{2RichInsideCell}
  \bigg|\pts_2(\curves_{\Omega})\ \backslash\ \Big(\bigcup_{M\in\mathcal{M}_\Omega}\pts_2(\curves_{M,\Omega}) \cup \bigcup_{S\in\mathcal{S}}\pts_2(\curves_{S,\Omega})\Big)\bigg|\lesssim_D (nE^{-3})^{4/3+3\eps}.
 \end{equation}

Since the number of cells $\Omega$ is $O(E^4)$, if we sum \eqref{2RichInsideCell} over all cells, then we get
 \begin{equation}
  \sum_{\Omega}\bigg|\pts_2(\curves_{\Omega})\ \backslash\ \Big(\bigcup_{M\in\mathcal{M}_\Omega}\pts_2(\curves_{M,\Omega}) \cup \bigcup_{S\in\mathcal{S}}\pts_2(\curves_{S,\Omega})\Big)\bigg| 
\lesssim_D E^{- 9 \eps} n^{4/3 + 3 \eps}.   \end{equation}

Now provided that we choose $E=O_{D,\eps}(1)$ sufficiently large, we get
  
   \begin{equation}\label{2RichInsideCellSummed}
  \sum_{\Omega}\bigg|\pts_2(\curves_{\Omega})\ \backslash\ \Big(\bigcup_{M\in\mathcal{M}_\Omega}\pts_2(\curves_{M,\Omega}) \cup \bigcup_{S\in\mathcal{S}}\pts_2(\curves_{S,\Omega})\Big)\bigg|   \leq \frac{1}{100}n^{4/3+3\eps},
 \end{equation}

 We will fix a value of $E$ so that \eqref{2RichInsideCellSummed} holds. Note that $Z(P)$ is a union of $O_{D,\eps}(1)$ irreducible varieties, each of which has degree at most $E=O_{D,\eps}(1)$. We will choose $C^\prime=O_{D,\eps}(1)$ so that $C^\prime\geq E$. 
 
Let $\mathcal{M}_1$ be the union of $\mathcal{M}_{\Omega}$ over all $\Omega$ together with all the irreducible components of $Z(P)$.  Let $\mathcal{S}_1=\bigcup_{\Omega}\mathcal{S}_{\Omega}$.  These will be the first of several intermediate objects we will consider before finally closing the induction. Note that with $E=O_{D,\eps}(1)$ fixed, we have
\begin{equation}
\begin{split}
|\mathcal{M}_1|&\lesssim_{D,\epsilon} n^{1/3-\epsilon},\\
|\mathcal{S}_1|&\lesssim_{D,\epsilon} n^{2/3-2\epsilon},
\end{split}
\end{equation}
and each $M\in\mathcal{M}_1$ has degree at most $C^\prime$.  By Equation \ref{2RichInsideCellSummed}, we also have

\begin{equation} \label{2richbound1} \bigg| P_2(\curves) \setminus \bigcup_{M \in \MMM_1} P_2(\curves_M) \cup \bigcup_{S \in \SSS_1} P_2(\curves_S) \bigg| \le
 \frac{1}{100} n^{4/3 + 3 \eps}. \end{equation} 

The bound on two-rich points in Equation \ref{2richbound1} is strong enough to close the induction, but the sets $\mathcal{M}_1$ and $\mathcal{S}_1$ are too big.  We need to process these sets to find smaller sets of varieties that still do a comparable job of covering two-rich points.  

\subsection{Dealing with three-dimensional surfaces}
For each $M,M^\prime\in \mathcal{M}_1$, let $S_{M,M^\prime}=M\cap M^\prime;$ we have that $S_{M,M^\prime}$ has degree $O_{E,\epsilon}(1)$. Define 
$$
\mathcal{S}_2 =  \{ S_{M,M^\prime}\colon M, M^\prime\in\mathcal{M}_1\}. 
$$
Then
$$
|\mathcal{S}_2|\leq |\mathcal{M}_1|^2 \lesssim_{D,\epsilon} n^{2/3-2\epsilon}.
$$

(We remark that the two-dimensional varieties in $\mathcal{S}_2$ may have degree more than $100 D^2$ and they may be reducible.  These are both problems for closing the induction.  We will deal with these issues in the next subsection, when we process the two-dimensional varieties.)

For each $M\in\mathcal{M}_1$, define 
$$
\lines_M^* = \lines_M \backslash  \bigcup_{S\in\mathcal{S}_2} \lines_S.
$$

We know that $\curves_M \subset \curves_M^* \cup  \bigcup_{S\in\mathcal{S}_2} \lines_S$.   Many of the points of $P_2(\curves_M)$ lie in $P_2(\curves_M^*) \cup  \bigcup_{S\in\mathcal{S}_2} P_2(\lines_S)$, but not all of them do.  Let $P_{2, hybrid}(\curves_M)$ be the set of 2-rich points of $\curves_M$ that are not contained in $P_2(\curves_M^*) \cup  \bigcup_{S\in\mathcal{S}_2} P_2(\lines_S)$.  We will prove the following bound on the number of these points:

\begin{equation}  \label{hybridpointbound}
\bigg| \bigcup_{M \in \mathcal{M}_1} P_{2, hybrid}(\curves_M) \bigg| \lesssim_{D,\epsilon} n^{4/3}. 
\end{equation}

Let $x$ be a point of $P_{2, hybrid} (\curves_M)$.  We know that $x$ lies in two curves $\gamma_1, \gamma_2 \in \curves_M$.  We know that at most one of $\gamma_1, \gamma_2$ lies in $\curves_M^*$.  Without loss of generality, suppose that $\gamma_2 \notin \curves_M^*$.  Therefore, $\gamma_2$ must lie in some variety $M_2 \in \MMM_1$, with $M_2 \not= M$.  But $\gamma_1$ cannot lie in $M_2$, or else $x$ would lie in $P_2( \curves_{\SSS_{M, M_2}})$.  The number of intersection points between $\gamma_1$ and varieties $M_2 \in \MMM_1$ that don't contain $\gamma_1$ is $\lesssim_D | \MMM_1 | \lesssim_{D,\epsilon} n^{1/3 - \eps}$.  Taking the union over all $M \in \mathcal{M}_1$, we get the bound \eqref{hybridpointbound}.  In other words, 

$$ \left| \bigcup_{M \in \mathcal{M}_1} P_2(\curves_M) \setminus \bigcup_{M \in \mathcal{M}_1} P_2(\curves_M^*) \cup \bigcup_{S \in \mathcal{S}_2} P_2(\curves_S) \right| \lesssim_{D,\epsilon} n^{4/3}. $$

Define
$$
\mathcal{M}_2 = \{M\in\mathcal{M}_1 \colon |\lines^*_M|> 2n^{2/3+\epsilon} \}.
$$
By construction, if $M,M^\prime\in\mathcal{M}_2$, then $\lines^*_M \cap \lines^*_{M^\prime}=\emptyset$. Thus 
\begin{equation}\label{sizeOfM2}
|\mathcal{M}_2|\leq \frac{1}{2}n^{1/3-\epsilon}.
\end{equation}

For each three-dimensional variety $M\in\mathcal{M}_1\backslash \mathcal{M}_2$, apply Proposition \ref{generalizedGuth} with parameter $\epsilon/2$ to the set of curves $\curves_M^*$ in the 3-dimensional variety $M$, and let $\mathcal{S}_M$ be the resulting collection of two-dimensional surfaces. Note that $|\mathcal{S}_M|\leq |\curves_M^*|^{1/2-\epsilon/2}\leq 2 n^{1/3+\epsilon/2}$, and each surface $S\in\mathcal{S}_M$ has degree $O_{D,\epsilon}(1)$. Finally,
$$
\bigg|\pts_2(\lines^*_M)\ \backslash\bigcup_{S\in\mathcal{S}_M}\pts_2(\lines_S))\bigg|\lesssim_{D,\epsilon}|\lines^*_M|^{3/2+\epsilon/2}\lesssim_{D,\epsilon}n^{1+3\epsilon}.
$$

Define
$$
\mathcal{S}_3 = \mathcal{S}_1\cup\mathcal{S}_2\ \cup \bigcup_{M\in\mathcal{M}_1\backslash \mathcal{M}_2}\!\! \mathcal{S}_M.
$$

Summing over all $M \in \mathcal{M}_1 \backslash \mathcal{M}_2$, and noting that $| \mathcal{M}_1 | \lesssim_{D,\epsilon} n^{1/3 - \eps}$, we conclude that

\begin{equation} \label{P_2smallM_1bound}
\bigg|\bigcup_{M\in\mathcal{M}_1\backslash \mathcal{M}_2} P_2(\curves_M^*) \setminus \bigcup_{S \in \mathcal{S}_3} P_2(\curves_S) \bigg| \lesssim_{D,\epsilon} n^{1/3 - \eps} \cdot n^{1 + 3 \eps} \lesssim_{D,\epsilon} n^{4/3 + 2 \eps}.
\end{equation}

Combining our equations so far, we see that

\begin{equation}
\bigg| \bigcup_{M \in \mathcal{M}_1} P_2(\curves_M)\  \backslash \bigcup_{M \in \mathcal{M}_2} P_2(\curves_M) \cup \bigcup_{S \in \mathcal{S}_3} P_2(\curves_S) \bigg| \lesssim_{D,\epsilon} n^{4/3 + 2 \eps},
\end{equation}

and so

\begin{equation} \label{2richbound2} \bigg| P_2(\curves) \setminus \bigcup_{M \in \MMM_2} P_2(\curves_M) \cup \bigcup_{S \in \SSS_3} P_2(\curves_S) \bigg| \le
 \frac{1}{100} n^{4/3 + 3 \eps} + C(D, \eps) n^{4/3 + 2 \eps}. \end{equation} 

By choosing $n$ sufficiently large, we can arrange that the size of this set is at most $\frac{2}{100} n^{4/3 + 3 \eps}$.  So we have effectively replaced $\mathcal{M}_1$ by the smaller set of 3-dimensional varieties $\mathcal{M}_2$.  The number of varieties in $\mathcal{M}_2$ is at most $\frac{1}{2}n^{1/3-\epsilon}$.  This is good enough to close the induction, and it even leaves us some room to add more three-dimensional varieties later.  The cost of this operation was to add more two-dimensional surfaces, replacing $\SSS_1$ by $\SSS_3$.  We can also bound the size of $|\mathcal{S}_3|$ by

\begin{equation}\label{sizeOfS2}
|\mathcal{S}_3| \leq |\mathcal{S}_1|+|\mathcal{S}_2|+n^{1/3-2\epsilon}|\mathcal{M}_1|\lesssim_{D,\epsilon}n^{2/3-2\epsilon}.
\end{equation}

Our bound for $|\mathcal{S}_3|$ is still too big to close the induction (but it is not much worse than our bound for $|\mathcal{S}_1|$).  In the next subsection, we process the two-dimensional varieties in $\mathcal{S}_3$.

\subsection{Dealing with two-dimensional surfaces}

In the last subsection, we constructed a set $\mathcal{S}_3$ of two-dimensional varieties.  In this subsection, we process the set $\mathcal{S}_3$ to close the induction.  There are four problems that we have to fix:

\begin{enumerate}
\item Surfaces in $\mathcal{S}_3$ can be reducible.

\item We know that each surface in $\SSS_3$ has degree $O_{D, \eps}(1)$, but the degree can be more than $100 D^2$.

\item A surface $S \in \mathcal{S}_3$ may not contain $n^{1/3 + 2 \eps}$ curves of $\curves$.  

\item $| \mathcal{S}_3 |$ is too big.

\end{enumerate}

We will fix these problems one at a time.  The last problem is the hardest and most important.

Each surface $S \in \mathcal{S}_3$ is a union of irreducible components

$$ S = S_1 \cup ... \cup S_l. $$

We let $\mathcal{S}_4$ be the set of irreducible components of surfaces $S \in \mathcal{S}_3$.  We know that each surface in $\SSS_3$ has degree $O_{D,\epsilon}(1)$, and so $l = O_{D,\epsilon}(1)$, and $| \SSS_4 | \lesssim_{D,\epsilon} | \SSS_3 | \lesssim_{D,\epsilon} n^{2/3 + 2 \eps}$.  For each $S \in \SSS_3$, we want to understand

$$ P_2(\curves_S)\ \backslash\ \bigcup_{i=1}^l P_2(\curves_{S_i}). $$

If a point $x$ lies in this set, then we must have $x \in \gamma, \gamma'$, where $\gamma \subset S_i$, $\gamma' \subset S_{i'}$, and $\gamma$ is not contained in $S_{i}$.  For a given $\gamma \in \curves_S$, the number of such points is $O_{D}(1)$.  Therefore, 

$$ \bigg| P_2(\curves_S)\ \backslash\ \bigcup_{i=1}^l P_2(\curves_{S_i}) \bigg| \lesssim_D | \curves_S | \lesssim_{D,\epsilon} \sum_{i=1}^l |\curves_{S_i} |.  $$

Taking a union over all $S \in \SSS_3$, we see that

$$ \bigg| \bigcup_{S \in \SSS_3} P_2(\curves_S)\ \backslash  \bigcup_{S \in \SSS_4} P_2(\curves_S) \bigg| \lesssim_{D,\epsilon} \sum_{S \in \SSS_4} |\curves_S|. $$

Next we bound this sum.  We know that $|\SSS_4| \lesssim_{D,\epsilon} n^{2/3 - 2 \eps}$.  We decompose $\SSS_4 = \SSS_{4, big} \cup \SSS_{4, small}$, where $S \in \SSS_{4,big}$ if $| \curves_S | > n^{2/3}$.  We apply Lemma \ref{unionsOfSurfacesLem}.  We can assume that $n$ is sufficiently large so that $n^{2/3} \ge C_1 |\SSS_4|$, and so Lemma \ref{unionsOfSurfacesLem} gives

$$ \sum_{S \in \SSS_{4, big}} |\curves_S| \le 2 \bigg| \bigcup_{S \in \SSS_{4,big}} \curves_S \bigg| \le 2 n. $$

On the other hand,

$$ \sum_{S \in \SSS_{4, small}} |\curves_S| \le | \SSS_4| n^{2/3} \lesssim_{D,\epsilon} n^{4/3}. $$

So all together, we have

\begin{equation} \label{sumL_S}
\sum_{S \in \SSS_4} |\curves_S| \lesssim_{D,\epsilon} n^{4/3}. 
\end{equation}

Plugging this bound into our equation above, we get

$$ \bigg| \bigcup_{S \in \SSS_3} P_2(\curves_S)\ \backslash  \bigcup_{S \in \SSS_4} P_2(\curves_S) \bigg| \lesssim_{D,\epsilon} n^{4/3}. $$

Next we deal with the degrees of the surfaces.  We decompose $\SSS_4 = \SSS_{4, high} \cup \SSS_{4, low}$, where $S \in \SSS_{4, low}$ if and only if the degree of $S$ is at most $100 D^2$.  By Corollary \ref{manyTwoRichDecomposition}, we know that for each $S \in \SSS_{4, high}$, $|P_2(\curves_S)| \lesssim_{D,\epsilon} | \curves_S|$.  Therefore, (using Equation \ref{sumL_S}), 

$$ \sum_{S \in \SSS_{4, high}} | P_2(\curves_S) | \lesssim_{D,\epsilon} \sum_{S \in \SSS_4} | \curves_S | \lesssim_{D,\epsilon} n^{4/3}. $$

Next we prune out surfaces containing few curves.  We let $C_2=O(1)$ be a large constant to be chosen later, and we define

$$ \SSS_5 := \{ S \in \SSS_4 \textrm{ so that } | \curves_S | > C_2 n^{1/3 + 2 \eps} \}. $$

We bound the contribution of the surfaces in $\SSS_4 \setminus \SSS_5$:

$$ \sum_{S \in \SSS_4 \setminus \SSS_5} | P_2(\curves_S)| \le \sum_{S \in \SSS_4 \setminus \SSS_5} |\curves_S|^2 \le
| \SSS_4 | \left( C_2 n^{1/3 + 2 \eps} \right)^2 \lesssim_{D,\epsilon} n^{4/3 + 2 \eps}. $$

To summarize, each surface $S \in \SSS_5$ is irreducible, with degree at most $100 D^2$, and contains at least $C_2 n^{1/3 + 2 \eps}$ curves of $\curves$.  Also we have shown that

$$ \bigg| \bigcup_{S \in \SSS_3} P_2(\curves_S)\ \backslash  \bigcup_{S \in \SSS_5} P_2(\curves_S) \bigg| \lesssim_{D,\epsilon} n^{4/3 + 2 \eps}. $$

Now we come to the most difficult and important issue: $\SSS_5$ may contain too many surfaces.  If we somehow knew that the sets $\{ \curves_S \}_{S \in \SSS_5}$ were disjoint, then since each $\curves_S$ has size at least $C_2 n^{1/3 + 2 \eps}$, it would follow that $| \SSS_5 | \le C_2^{-1} n^{2/3 - 2 \eps}$, which would be good enough to close the induction.  Since we can select $C_2=O(1)$ to be large, it would suffice if the sets $\{ \curves_S \}_{S \in \SSS_5}$ were merely ``roughly disjoint.''  But these sets can fail badly to be disjoint.  In particular, this can happen if many surfaces of $\SSS_5$ cluster into a low-degree 3-dimensional variety $M$.  We need to recognize when this is happening.  When it does happen, we add the variety $M$ to $\mathcal{M}$, and we delete the surfaces in $M$ from $\SSS_5$.  We will show that the remaining surfaces are roughly disjoint.  We find the relevant varieties $M$ by using Proposition \ref {GZForSurfaces}.  We apply Proposition \ref{GZForSurfaces} to $\SSS_5$.  The number of surfaces in $\SSS_5$ is $| \SSS_5 | \lesssim_{D,\epsilon} n^{2/3 - 2 \eps}$.  We apply Proposition \ref{GZForSurfaces} to $\SSS_5$ with the value $A$ taken as

$$ A = \left( n^{2/3 - 2 \eps} \right)^{\frac{1}{2} + \frac{\eps}{2}}. $$

\noindent If $n$ is large enough, then we see that $A \ge C_1 |\SSS_5|^{1/2}$, and so
Proposition \ref{GZForSurfaces} applies.  It tells us that there is a set $\MMM_3$ of irreducible 3-dimensional varieties with the following properties:

\begin{itemize}

\item The degree of each $M \in \MMM_3$ is $O_{D,\epsilon}(1)$.  We can choose the degree $C'$ in the statement of the main theorem so that the degree of each $M \in \MMM_3$ is at most $C'$.

\item $|\MMM_3| \lesssim (n^{2/3 - 2 \eps})^{(1/2 - \eps/2)} \lesssim_{D,\epsilon} n^{1/3 - \frac{10}{9} \eps}$.

\item For each $M \in \MMM_3$, we define $\SSS_M := \{ S \in \SSS_5 | S \subset M \}$.  Then for each $M \in \MMM_3$,

$$ | \SSS_M | \geq (n^{2/3 - 2 \eps})^{(1/2 + \eps/2)} \geq n^{1/3 - \frac{3}{4} \eps}. $$

\item Define $\SSS := \SSS_5 \backslash \bigcup_{M \in \MMM_3} \SSS_M$.  %Now define 

%$$\CCC_2(\SSS) := \{ \gamma \in \curves | \gamma \textrm{ is contained in at least two of the surfaces } S \in \SSS' \}. $$

Then 
$$
 \sum_{\gamma\in \curves} |\{S\in \SSS | \gamma\subset S\}| \lesssim_{D,\epsilon} |\SSS_5|^{3/2+\eps/2}\lesssim_{D,\epsilon} n^{1-\eps/4}.
$$
\end{itemize}

We can rephrase this last inequality as an incidence bound.  We let $I(\curves, \SSS)$ denote the set of pairs $(\gamma, S) \in \curves \times \SSS$ with $\gamma \subset S$.  The last inequality can be rewritten as

$$ | I(\curves, \SSS) | \lesssim_{D,\epsilon} n^{1-\eps/4}. $$

We have now defined our final set of surfaces $\SSS$.  Also, we can now define our final set of 3-dimensional varieties $\MMM$ by

$$ \MMM = \MMM_2 \cup \MMM_3. $$

To finish the proof, we have to check that $\SSS$ and $\MMM$ close the induction.  First we consider $\SSS$.  Since $\SSS \subset \SSS_5$, we already know that each surface $S \in \SSS$ is irreducible with degree at most $100 D^2$ and contains at least $C_2 n^{1/3 + 2 \eps}$ curves of $\curves$.  We now bound $| \SSS |$.  To do so, we double count the incidences $I(\curves, \SSS)$.  On the one hand, we know that each $S \in \SSS$ contains at least $C_2 n^{1/3 + 2 \eps}$ curves of $\curves$, and on the other hand we know that $| I(\curves, \SSS)| \lesssim_{D, \eps} n^{1 - \eps/4}$.  If $n$ is big enough, we get:

$$ C_2 n^{1/3 + 2 \eps} | \SSS | \le | I(\curves, \SSS) |  \le n. $$

Choosing $C_2 \ge 10$, we see that

$$ | \SSS | \le n^{2/3 - 2 \eps}. $$

Next we have to check that $\MMM$ obeys the desired properties.  We have to check that $|\MMM| \le n^{1/3 - \eps}$, and we have to check that each $M \in \MMM$ contains at least $n^{2/3 + \eps}$ curves of $\curves$.

$$ | \MMM | \le |\MMM_2 | + | \MMM_3 | \le \frac{1}{100} n^{1/3 - \eps} + C(D, \eps) n^{1/3 - \frac{10}{9} \eps}. $$

\noindent Since we can assume $n$ is sufficiently large, we get $| \MMM | \le n^{1/3 - \eps}$.  

We already know that each $M \in \MMM_2$ contains at least $2 n^{2/3 + \eps}$ curves of $\curves$.  If $M \in \MMM_3$, we know that $M$ contains at least $n^{1/3 - \frac{3}{4} \eps}$ surfaces $S \in \SSS_5$.  Each surface $S \in \SSS_5$ contains at least $C_2 n^{1/3 + 2 \eps}$ curves of $\curves$.  By Corollary \ref{corUnionsOfSurfacesLem},

$$ | \curves_M | \gtrsim_D \min \left( (n^{1/3 + 2 \eps})^2, n^{1/3 + 2 \eps} \cdot n^{1/3 - \frac{3}{4} \eps} \right) \geq n^{2/3 + \frac{5}{4} \eps}. $$

\noindent Since we can assume $n$ is sufficiently large, we get $| \curves_M | \ge n^{2/3 + \eps}$.  

Finally, combining all of our estimates about two-rich points in the two subsections, we see that

$$ \bigg| P_2(\curves) \setminus \bigcup_{M \in \MMM} P_2(\curves_M) \cup \bigcup_{S \in \SSS} P_2(\curves_S) \bigg| \le
 \frac{1}{100} n^{4/3 + 3 \eps} + C(D, \eps) n^{4/3 + 2 \eps}. $$
 
Since we can assume $n$ is sufficiently large, this closes the induction and finishes the proof of Theorem \ref{twoRichPtsThm}.

\section{Polynomial partitioning over $\CC$: complex contemplations and real realities}
In this section we will discuss several half proofs and non-results. These are proof ideas that appear promising but turn out to be fatally flawed.
\subsection{The dream: polynomial partitioning over $\CC$}

Let $P\in\CC[z_1,\ldots,z_d]$ be a complex polynomial. Define $\operatorname{Re}P\colon \RR^{2d}\to \RR$ by $\operatorname{Re}P(x_1,y_1,\ldots,x_d,y_d)=\operatorname{Re}(P(x_1+iy_1,\ldots,x_d+iy_d))$. Define $\operatorname{Im}P$ similarly. In particular, $Z(\operatorname{Re}P)$ and $Z(\operatorname{Im}P)$ are real hypersurfaces in $\RR^{2d}$ of degree $\deg(P).$ Let $\iota\colon\CC^d\to\RR^{2d}$ be the usual identification of $\CC$ with $\RR^2$. 
\begin{conj}\label{complexPartitioning}
Let $\pts\subset\CC^d$ be a set of $n$ points. Then for each $E\geq 1$, there is a polynomial $P\in\CC[z_1,\ldots,z_d]$ of degree at most $E$ so that each connected component of $Z(\operatorname{Re}(P))\subset\RR^{2d}$ contains $O(n d^{-E})$ points from $\iota(\pts)$. Similarly, each connected component of $Z(\operatorname{Im}(P))\subset\RR^{2d}$ contains $O(n d^{-E})$ points from $\iota(\pts)$
\end{conj}
Conjecture \ref{complexPartitioning} appears plausible at first, because the dimension of the vector space of degree $E$ polynomials in $\CC[z_1,\ldots,z_d]$ is $\binom{E+d}{d}$. In particular, given $\binom{E+d}{d}$ sets of points $\pts_1,\ldots,\pts_{\binom{E+d}{d}}$ in $\CC^d$, it is possible to find a complex polynomial $P$ so that $\{ \operatorname{Re}(P)>0\}$ and  $\{ \operatorname{Re}(P)<0\}$ contain an equal number of points from each of the sets $\pts_1,\ldots,\pts_{\binom{E+d}{d}}$. As we will discuss in Section \ref{conjIsFalse} below, however, it appears that Conjecture \ref{complexPartitioning} is likely false. For the moment though, we will suspend disbelief and see what the implications of Conjecture \ref{complexPartitioning} would be.
\subsection{An overly optimistic proof of Szemer\'edi-Trotter in the complex plane}\label{optimisticComplexST}
If Conjecture \ref{complexPartitioning} were true, it would allow for an elementary proof of the complex Szemer\'edi-Trotter theorem. The key observation is that if $Q\in\CC[z]$ is a complex polynomial of degree $E$, then $\operatorname{Re}(Q)\colon\RR^2\to\RR$ is harmonic, and thus $Z(\operatorname{Re}(Q))$ contains $O(E)$ connected components. Similarly $Z(\operatorname{Im}(Q))$ contains $O(E)$ connected components. This means that if $P\in\CC[z_1,z_2]$ is a degree $E$ polynomial, and if $L\subset\CC^2$ is a complex line, then $\iota(L)$ intersects $O(E)$ connected components of $Z(\operatorname{Re}(P))$, and $\iota(L)$ intersects $O(E)$ connected components of $Z(\operatorname{Im}(P))$. 

Let $\pts\subset\CC^2$ be a set of $m$ points, and let $\mathcal{L}$ be a set of $n$ complex lines; assume that $n^{1/2}\leq m\leq n^2$. Use Conjecture \ref{complexPartitioning} to find a polynomial $P\in\CC[z_1,z_2]$ of degree $E=m^{2/3}n^{-1/3}$ so that each connected component of $Z(\operatorname{Re}(P))$ contains $O(mE^{-2}=O(m^{-1/3}n^{2/3})$ points from $\iota(\pts)$, and similarly each connected component of $Z(\operatorname{Im}(P))$ contains $O(m^{-1/3}n^{2/3})$ points from $\iota(\pts)$.

Let $m_{\Omega}$ be the number of points from $\iota(\pts)$ and let $n_{\Omega}$ be the number sets from $\{\iota(L)\colon L\in\mathcal{L}\}$ that meet the cell $\Omega$. We have
\begin{equation*}
\begin{split}
I(\iota(\pts)\backslash Z(\operatorname{Re}(P)), \iota(\mathcal L))&=\sum_{\Omega}m_{\Omega}n_{\Omega}^{1/2}+\sum_{\Omega}n_{\Omega}\\
&\leq C\Big(\sum_{\Omega}m_{\Omega}^2\Big)^{1/2}\Big(\sum_{\Omega}n_{\Omega}\Big)^{1/2}+C\sum_{\Omega}n_{\Omega}\\
%&\lesssim (m/D)(Dn)^{1/2}+Dn\\
&=O(m^{2/3}n^{2/3}).
\end{split}
\end{equation*}

Similarly,
\begin{equation*}
I(\iota(\pts)\backslash Z(\operatorname{Im}(P)), \iota(\mathcal L)) =O(m^{2/3}n^{2/3}). 
\end{equation*}

This implies that
\begin{equation}\label{incidencesInCell}
I(\pts\backslash Z(P), \mathcal{L}) = O(m^{2/3}n^{2/3}).
\end{equation}

Finally, let $\mathcal{L}_1=\{L\in\mathcal L\colon L\subset Z(P)\}$. We have $|\mathcal{L}_1|=O(m^{2/3}n^{-1/3})$. Thus
\begin{equation}\label{incidencesLinesInVariety}
|I(\pts,\mathcal{L}_1)|=O( m^{1/2}|\mathcal{L}_1|)=O(m^{5/6}n^{-1/3})=O(m^{2/3}n^{1/3}).
\end{equation}
Since each line $L\in\mathcal{L}\backslash\mathcal{L}_1$ meets $Z(P)$ in at most $O(D)=O(m^{2/3}n^{-1/3})$ points, we have
\begin{equation}\label{incidencesLinesNotInVariety}
|I(\pts\cap Z(P),\mathcal{L}\backslash \mathcal{L}_1)|=O( m^{2/3}n^{2/3}).
\end{equation}
The theorem follows from combining \eqref{incidencesInCell}, \eqref{incidencesLinesInVariety}, and \eqref{incidencesLinesNotInVariety}. Note that this proof is much simpler than the two existing proofs due to T\'oth \cite{To} and the second author \cite{Z}.
\subsection{An optimistic bound for two-rich points in higher dimensions}
If Conjecture \ref{complexPartitioning} were true, it would allow us to prove a complex analogue of Theorem \ref{twoRichPtsThm} in higher dimensions. In short, 
\begin{conj}\label{generalDimTwoRichPtBd}
Let $Z\subset\CC^d$ be a bounded-degree irreducible variety of dimension $d^\prime$. Let $\mathcal{L}$ be a set of $n$ low degree curves contained in $Z$. Then for each $j=2,\ldots,d-1$, there exists a set $\mathcal{S}_j$ of low degree $j$ dimensional irreducible varieties, such that for each index $j$, $|\mathcal{S}_j|\leq n^{(d-j)/(d-1)-(d-j)\epsilon}$; each variety $S\in\mathcal{S}_j$ contains at least $n^{(j-1)/(d-1)+(d-j)\epsilon}$ of the curves; and
\begin{equation}
\Big|\pts_2(\mathcal{L})\ \ \backslash\ \ \bigcup_{j=2}^{d-1}\bigcup_{S\in\mathcal{S}_j}\pts_2(S(\mathcal{L}))\Big|\leq n^{d/(d-1)+d\epsilon}.
\end{equation}
\end{conj}
If Conjecture \ref{complexPartitioning} were true, then Conjecture \ref{generalDimTwoRichPtBd} could be proved using a similar strategy to the proof of Theorem \ref{twoRichPtsThm}. The proof of Theorem \ref{complexPartitioning} begins by partitioning $\RR^d$ into cells. If Conjecture  \ref{complexPartitioning} were true, we could perform an analogous partition on $\CC^{d^\prime}$ (considered as a subset of $\CC^d$) instead. (In Theorem \ref{twoRichPtsThm} we performed a polynomial partition adapted to the set of curves, but this is merely a technical convenience; see \cite{G} for details on why it suffices to be able to partition a set of points.)

The proof of Theorem \ref{twoRichPtsThm} used two main ingredients. First, it used a bound on the number of two-rich points determined by a collection of curves in a bounded-degree three-dimensional surface in $\RR^4$. This is essentially a lower dimensional version of the original problem, so if we prove the result by induction on the dimension $d$, then we can assume that such lower-dimensional bounds already exist. 

Second, the proof of Theorem \ref{twoRichPtsThm} used a bound on the number of two-rich curves determined by a collection of bounded-degree two-dimensional surfaces in $\RR^4$. In order to prove this second bound, it is temping to try to ``slice'' the surface arrangement with a generic real hyperplane. Then, one would hope, two-dimensional surfaces become one-dimensional curves, and two-rich curves become two-rich points. One could then apply an existing (lower dimensional) bound on the number of two-rich points. Unfortunately, this does not work because it may be impossible to find a hyperplane in $\RR^d$ that intersects each of the two-rich curves. If we work over $\CC$, however, this problem disappears---it is possible to find a hyperplane that meets each two-rich curve in a two-rich point. Thus Conjecture \ref{generalDimTwoRichPtBd} would likely be achievable. As we will see below, however, Conjecture \ref{complexPartitioning} is almost certainly false.
\subsection{Why Conjecture \ref{complexPartitioning} probably isn't true}\label{conjIsFalse}
The following lemma demonstrates why Conjecture \ref{complexPartitioning} is not true as stated. It also suggests that any reasonable re-formulation of Conjecture \ref{complexPartitioning} is also likely doomed to failure.
\begin{lemma}
Let $P\in\CC[z_1,\ldots,z_d]$ be a polynomial of degree $E$. Then $\RR^{2d}\backslash Z(\operatorname{Re}(P))$ has at most $2E$ connected components. In particular,  if $\pts\subset\CC^d$ is a finite set of points, then it is impossible for each connected component of $\RR^{2d}\backslash Z(\operatorname{Re}(P))$ to contain fewer than $|\pts|/(2E)$ points from $\iota(\pts)$. 
\end{lemma}

\begin{proof}
Let $f=\operatorname{Re}(P)$. Let $L$ be a complex line in $\CC^d$.  $P$ is holomorphic on $L$, and thus $f$ is harmonic on $\iota(L)$.  This means that every connected component of $\iota(L)\backslash Z(f)$ is unbounded.  Now, pick a complex line $L_0$ through the origin so that the highest order part of $P$ does not vanish on $L_0$.  If $L_0$ is parametrized by $w$, then on $L_0$, $P= c_E w^E + O(|w|^{E-1})$ as $|w|\to\infty$. If $x\in \CC^d$, let $L_x$ be the line passing through $x$ parallel to $L_0$. If $L_x$ is parameterized by $x+w$, then on $L_w$, we have $P= c_E w^E + O_{|x|}(|w|^{E-1})$  as $|w|\to\infty$. In particular, for each $r_1,r_2>0$, there is a number $R>0$ so that if $x\in B_{r_1}$, then on each connected component of $L_x\cap S_R$, there is a point $z\in L_x\cap S_R$ so that $f\neq 0$ for each point of $B(z,r_2)$; here $B_r=\{z\in\CC^d\colon |z|\leq r\}$ is the ball of radius $r$ and $S_R=\{z\in\CC^d\colon |z|=R\}$ is the sphere of radius $R$.

Consider the circle $L_0\cap S_R$. If $R$ is large then $Z(f)$ cuts this circle into $2 E$ pieces. We claim that we can move any point $x$ in $\RR^{2d}$ into one of these $2 E$ pieces without crossing $Z(f)$.

To see that this is true, let $x$ be a point in $B_{r}\backslash Z(f)$ for some $r$. If we select $R$ sufficiently large (depending on $r$), then on each connected component of $L_x\cap S_R$, there is a point $z\in L_x\cap S_R$ so that $f\neq 0$ for each point of $B(z,2r)$. Note that $B(z,2r)$ must intersect $L_0\cap S_R$. Thus $z$ lies on the same connected component as some segment of $L_0\cap S_R\cap Z(f)$. Now, by the discussion above, $x$ lies on the same connected component as one of the segments of $L_x\cap S_R\cap Z(f)$, and this segment is part of the same connected component as one of the segments of $L_0\cap S_R\cap Z(f)$; there are only $2E$ segments of this type. Thus for each $r>0$, the set $B_{r}\backslash Z(f)$ contains at most $2E$ connected components. Since this holds for all $r>0$, we conclude that $\RR^{2d}\backslash Z(f)$ contains at most $2E$ connected components.
\end{proof}
\bibliographystyle{abbrv}
\bibliography{richPointsBiblio}

\begin{thebibliography}{10}

\bibitem{BB}
S.~Barone and S.~Basu.
\newblock Refined bounds on the number of connected components of sign
  conditions on a variety.
\newblock {\em Discrete Comput. Geom.}, 47(3):577--597, 2012.

\bibitem{BCR}
J.~Bochnak, M.~Coste, and M.-F. Roy.
\newblock {\em Real algebraic geometry}.
\newblock Springer-Verlag, Berlin, 1998.

\bibitem{SZ}
J.~Ellenberg, J.~Solymosi, and J.~Zahl.
\newblock Curve-curve tangencies and orthogonalities.
\newblock {\em Discrete Analysis}, 22:1--22, 2016.

\bibitem{G}
L.~Guth.
\newblock Distinct distance estimates and low degree polynomial partitioning.
\newblock {\em Disc. Comput. Geom.}, 53(2):428--444, 2015.

\bibitem{G2}
L.~Guth.
\newblock Polynomial partitioning for a set of varieties.
\newblock {\em Math. Proc. Camb. Phil. Soc.}, 159:459--469, 2015.

\bibitem{GK}
L.~Guth and N.~Katz.
\newblock On the {E}rd{\H{o}}s distinct distance problem in the plane.
\newblock {\em Ann. of Math.}, 181:155--190, 2015.

\bibitem{GZ}
L.~Guth and J.~Zahl.
\newblock Algebraic curves, rich points, and doubly-ruled surfaces.
\newblock {\em arXiv:1503.02173}, 2015.

\bibitem{Harris}
J.~Harris.
\newblock {\em Algebraic geometry: {A} first course}, volume 133 of {\em
  Graduate Texts in Mathematics}.
\newblock Springer-Verlag, New York, 1995.

\bibitem{K}
J.~Koll\'ar.
\newblock Szemer\'edi-{T}rotter-type theorems in dimension 3.
\newblock {\em Adv. Math.}, 271:30 -- 61, 2015.

\bibitem{SS}
M.~Sharir and N.~Solomon.
\newblock Incidences between points and lines in {$\mathbb{R}^4$}.
\newblock {\em Discrete Comput. Geom.}, 57(3):702--756, 2017.

\bibitem{ShZ}
E.~Szab\'o, A.~Sheffer, and J.~Zahl.
\newblock Point-curve incidences in the complex plane.
\newblock {\em arXiv:1502.07003}, 2015.

\bibitem{To}
C.~T\'oth.
\newblock The {S}zemer\'edi-{T}rotter theorem in the complex plane.
\newblock {\em Combinatorica}, 35(1):95--126, 2015.

\bibitem{Z}
J.~Zahl.
\newblock A {S}zemer\'edi-{T}rotter type theorem in $\mathbb{R}^4$.
\newblock {\em Discrete. Comput. Geom}, 54(3):513--572, 2015.

\bibitem{Z2}
J.~Zahl.
\newblock A note on rich lines in truly high dimensional sets.
\newblock {\em FoM, Sigma}, 4(e2):1--13, 2016.

\end{thebibliography}

\end{document}